\documentclass{amsart}
\usepackage{amsfonts,amssymb,amsmath,amsthm}
\usepackage{url}
\usepackage{enumerate}

\newtheorem{thrm}{Theorem}[section]
\newtheorem{lem}[thrm]{Lemma}

\newtheorem{cor}[thrm]{Corollary}
\theoremstyle{definition}

\numberwithin{equation}{section}

\def\R{{\mathbb R}}
\def\C{{\mathbb C}}
\def\e{\emptyset}
\def\di{\displaystyle}

\def\ep{\epsilon}

\def\saa{\Sigma_A^+}

\def\ff{{\mathcal F}}
\def\ep{\epsilon}
\def\spec{\mbox{\rm spec}\,}

\def\cc{{\mathcal C}}

\def\ms{\medskip}
\def\var{\mbox{\rm var}}
\def\be{\begin{equation}}
\def\ee{\end{equation}}

\def\tg{\tilde{g}}

\author{Luchezar Stoyanov}
\address{
School of Mathematics and Statistics,\\
University of Western Australia, Crawley WA 6009,  Australia}
\email{luchezar.stoyanov@uwa.edu.au}

\begin{document}

\title[Ruelle transfer operators]{On Gibbs measures and spectra\\ of Ruelle transfer operators}

\begin{abstract}
We prove a comprehensive version of the Ruelle-Perron-Frobenius Theorem 
with explicit estimates of the spectral radius of the Ruelle transfer operator and various other
quantities related to spectral properties of this operator. The novelty here is that the H\"older
constant of the function generating the operator appears only polynomially, not exponentially as
in previous known estimates.
\end{abstract}
\maketitle

\section{Introduction} \label{sect1}
\renewcommand{\theequation}{\arabic{section}.\arabic{equation}}

We consider an one-sided shift space
$$\Sigma_A^+ = \{ \xi = (\xi_0, \xi_1, \ldots,\xi_m,\ldots) : 1\leq \xi_i \leq q, A(\xi_i,\xi_{i+1}) = 1
\; \mbox{\rm for all } \; i\geq 0\;\}\;,$$
where $A$ is a $q\times q$ matrix of $0$'s and $1$'s ($q\geq 2$). We assume that $A$ is 
{\it aperiodic}, i.e.  there exists an integer $M > 0$ such that  $A^M(i,j) > 0$ for all $i$, $j$ (see
e.g. Ch. 1 in \cite{PP}). The {\it shift map} $\sigma : \saa \longrightarrow \saa$ is defined by $\sigma(\xi) = \xi'$, 
where $\xi'_i = \xi_{i+1}$ for all $i \geq 0$. 

In this paper we consider {\it Ruelle transfer operators} $L_f : C(\saa) \longrightarrow C(\saa)$  defined by 
real-valued functions $f : \saa \longrightarrow \R$ by
$$L_fg(x) = \sum_{\sigma(y) = x} e^{f(y)}\, g(y)\;.$$ 
Here $C(\saa)$ denotes the space of all continuous functions $g : \saa \longrightarrow \R$ with the product topology.
Given $\theta \in (0,1)$, consider the metric $d_\theta$ on $\saa$ defined by
$d_\theta(\xi,\eta) = 0$ if $\xi = \eta$ and $d_\theta(\xi,\eta) = \theta^k$ if $\xi \neq \eta$ and $k\geq 0$
is the maximal integer with $\xi_i = \eta_i$ for $0\leq i \leq k$. For any function  $g: \saa \longrightarrow \R$ set
$$\var_k g = \sup\{ |g(\xi)- g(\eta)| : \xi_i = \eta_i ,\: 0\leq i \leq k\}\:\:\:, \:\:\:
|g|_\theta = \sup \left\{ \frac{\var_k g}{\theta^k} : k \geq 0\right\} \:,$$
$$|g|_\infty = \sup \{ |g(\xi)| : \xi\in\saa\}\:\:\:, \:\:\: \| g\|_\theta = |g|_\theta + |g|_\infty\;.$$ 
Denote by $\ff_\theta(\saa)$ the {\it space of all functions} $g$ on $\saa$ with $\|g\|_\theta < \infty$,
and by $\spec_\theta(L_g)$ the {\it spectrum} of $L_g : \ff_\theta(\saa) \longrightarrow \ff_\theta(\saa)$.

The Ruelle-Perron-Frobenius Theorem concerns spectral properties of the transfer operator
$L_f : \ff_\theta(\saa) \longrightarrow \ff_\theta(\saa)$. Assuming $A$ is aperiodic and $f \in \ff_\theta(\saa)$
is real-valued, it asserts that $L_f$ has a simple maximal positive eigenvalue $\lambda$, a
corresponding strictly positive eigenfunction $h$ and a probability measure $\nu$ on $\saa$
such that $\spec_\theta(L_f)\setminus \{\lambda\}$ is contained in a disk of radius 
$\rho\, \lambda$ for some $\rho\in (0,1)$,  $L_f^*\nu = \lambda \nu$, and assuming $h$ is normalized by
$\int h\, d\nu = 1$, we also have
\be
\lim_{n\to\infty} \frac{1}{\lambda^n}\, L^n_f g = h\, \int g\, d\nu 
\ee
for all $g \in \ff_\theta(\saa)$.  This was proved by Ruelle \cite{R1} (see also \cite{R2}). 
In the case of a complex-valued function $f$ similar results were established by Pollicott \cite{Po}. 

In this paper a comprehensive version of the Ruelle-Perron-FrobeniusTheorem is considered
which provides explicit estimates for the various constants and functions involved, e.g. the function
$h$ and the constant $\rho$ mentioned above, as well as the speed of convergence in (1.1).
Estimates of this kind were derived in \cite{St1}, however the constants that appeared there, including the estimate $\rho$
for the spectral radius of the operator $L_f$, involved terms of the form $e^{C |f|_\theta}$ for
various constants $C> 0$. The same applies to the estimates that appear in \cite{B}, \cite{R1}, \cite{R2}, \cite{PP}
and also to the estimate of the spectral radius of $L_f$ obtained in  \cite{N}. 

From our personal experience, when estimates for families of Ruelle transfer operators 
$L_f$ are considered for a class of functions $f$, usually the norms $|f|_\infty$ are uniformly bounded,
however the H\"older constants $|f|_\theta$ can vary a lot and in some cases can be arbitrarily large. 
That is why, estimates involving terms of the form   $e^{C |f|_\theta}$ are particularly unpleasant.  

All estimates obtained in this paper involve only powers of $|f|_\theta$, and,  in this sense, they are significantly 
sharper than the existing ones.

The motivation for \cite{St1} came from investigations in scattering theory  on distribution of scattering resonances,
in particular in dealing with the so called Modified Lax-Phillips Conjecture for obstacles $K$ in $\R^n$
that are finite disjoint unions of strictly convex bodies with smooth boundaries \cite{St2}. The present work 
stems from studies on decay of correlations for Axiom A flows and 
spectra of Ruelle transfer operators  in the spirit of \cite{D} and \cite{St3}.

Sect. 2 below contains the statement of the Ruelle-Perron-Frobenius Theorem 
with comprehensive estimates of the constants involved, while Sect. 3 is devoted to a proof of the theorem.  
As in \cite{St1}, we follow the main frame of the proof in \cite{B} with necessary modifications.

\section{The Ruelle-Perron-Frobenius Theorem}

\setcounter{equation}{0}

In what follows $A$ will be a $q\times q$ matrix ($q\geq 2$) such that $A^M > 0$ for some integer $M \geq 1$,
$\theta\in (0,1)$ will be a fixed number and $f \in  \ff_\theta(\saa)$ will be a fixed real-valued function. Set
\be
b = b_f = \max\{ 1, |f|_\theta\} .
\ee

\ms

\begin{thrm}\label{t1}
({\bf Ruelle-Perron-Frobenius Theorem}) 
   (a) {\it There exist a unique $\lambda = \lambda_f> 0$, a probability measure $\nu = \nu_f$ on $\saa$ and a
positive function $h = h_f \in \ff_\theta(\saa)$ such that $L_f h = \lambda\, h$ and $\di \int h\, d\nu = 1$.
The spectral radius of $L_f$ as an operator on $\ff_\theta(\saa)$ is $\lambda$, and 
its essential spectral radius  is $\theta\, \lambda$. The eigenfunction $h$ satisfies
\be
\frac{1}{K} \leq h \leq K \quad, \quad |h|_\theta \leq B b K ,
\ee
where 
\be
K = B b^{r_0} ,
\ee
and the constants $B$ and $r_0$ can be chosen as}
\be
B = \frac{e^{\frac{2\theta }{1-\theta}} q^{M+1} e^{2(M+1) |f|_\infty}}{1- \theta} \quad, \quad
r_0 = \frac{\log q + 2 |f|_\infty }{|\log \theta|} .
\ee

    (b) {\it  The probability measure $\hat{\nu} = h\,\nu$
(this is the so called {\it Gibbs measure} generated by $f$) is $\sigma$-invariant.}

\ms

   (c) {\it We have $\spec_\theta(L_f) \bigcap \{ z\in \C : |z| = \lambda\} = \{\lambda\}$. Moreover
$\lambda$ is a simple eigenvalue for $L_f$ and every $z\in \spec_\theta (L_f)$ with $|z| < \lambda$ satisfies 
$|z| \leq \rho \, \lambda$, where }
\be
\rho = 1- \frac{1-\theta}{8 K^3}  \in (0,1) .
\ee

\ms

(d) {\it  For every $g\in \ff_\theta(\saa)$ and every integer $n \geq 0$ we have
\be
\left\| L_f^n g - h \int g \, d \nu \right\|_\theta \leq  D_f \, \lambda^n\, \rho^n\, \|g\|_\theta ,
\ee
where }
$D_f =  \frac{100 K^5\, b^3}{1-\theta} .$
\end{thrm}

\ms

The constants $K$, $ \rho$, $D_f$, etc. are not optimal, slightly better estimates are possible as one
can see from the proof in Sect 3.

\section{Proof of Theorem 2.1}
\setcounter{equation}{0}

We will use  the notation and assumptions from Sect. 2. Set $L = L_f$.
Given $x = (x_0,x_1, \ldots)\in \saa$ and $m \geq 0$, consider the {\it cylinder of length} $m$ 
$$\cc_m[x] = \{ y = (y_0,y_1, \ldots)\in \saa : y_j = x_j\:\:\mbox{\rm for all }\:\: j = 0,1,\ldots,m\} $$
determined by $x$. Set
$g_m(x) = g(x) + g(\sigma x) + \ldots + g(\sigma^{m-1} x) .$

As in \cite{B}, it follows from the Schauder-Tychonoff Theorem that there exist a Borel probability measure $\nu$
on $\saa$ and a number $\lambda > 0$ such that  
$L^*_f \nu = \lambda\, \nu$, that is $\di \lambda \int g \, d\nu = \int L_f g\, d\nu$ for every $g \in C(\saa)$.
With $g = 1$ this gives $\di \lambda= \int L_f 1 \, d\nu$. Clearly, $(L_f 1)(x)  = \sum_{\sigma\xi = x} e^{f(\xi)} \leq q e^{|f|_\infty}$,
and also $(L_f 1)(x) \geq e^{-|f|_\infty}$ for all $x \in \saa$. Thus,
\be
e^{-|f|_\infty} \leq \lambda \leq q\, e^{|f|_\infty} .
\ee

Let $m_0 = m_0(f,\theta) \geq 1$ be {\it the integer such that}
\be
\theta^{m_0} < \frac{1}{b} \leq \theta^{m_0-1} .
\ee
Then $m_0  \, \log \theta <  - \log b \leq (m_0-1)  \log\theta$, so
$$m_0  -1 \leq \frac{\log b}{|\log \theta|} <  m_0.$$

The first significant difference between our argument and the one in \cite{B} is in the definitions
of the constants $B_m$ and the space $\Lambda$ below. In our argument they depend on $m_0$, i.e. on $b$. 

For $m \geq m_0$ set $ \di B_m =  e^{2\sum_{j = m-m_0+1}^\infty \theta^j}$ and define
$$\Lambda = \left\{ g\in C(\saa) : g \geq 0,  \int g \,d\nu = 1 ,  g(y) \leq B_m\, g(x) \:
\mbox{\rm whenever } \: y\in \cc_m[x] ,  m \geq m_0 \right\}.$$
Then $\di B_{m_0} = e^{2 \sum_{j=1}^\infty \theta^j} = e^{\frac{2\theta}{1-\theta}}.$ 
Notice that in the above definitions {\bf we only consider integers $m$ with $m \geq m_0$.}
This will be significant later on.

\begin{lem} \label{l1}
{\it $\Lambda$ is a non-empty, convex and closed  in $C(\saa)$ equicontinuous family of functions and
the operator} $\di T = \frac{1}{\lambda} L$ maps $\Lambda$ into  $\Lambda$.
\end{lem}

\begin{proof}
We use a modification of the proof of Lemma 1.8 in \cite{B}.

It is clear that $\Lambda$ is convex and closed in $C(\saa)$, and also $\Lambda \neq \e$ since $1 \in \Lambda$.

Consider arbitrary $x = (x_0,x_1, \ldots), z = (z_0,z_1, \ldots) \in \saa$ and $g \in \Lambda$. Since $A^M > 0$, there exists
a sequence $(z'_{m_0+1}, z'_{m_0+2}, \ldots, z'_{m_0+M-1}, z'_{m_0+M} = x_0)$
such that 
$$y = (z_0,z_1, \ldots, z_{m_0}, z'_{m_0+1}, z'_{m_0+2}, \ldots, z'_{m_0+M-1}, z'_{m_0+M} = x_0,
x_1, x_2, \ldots) \in \saa .$$
Then $d_\theta(y,z) \leq \theta^{m_0}$, so $g \in \Lambda$ implies $g(z) \leq B_{m_0} g(y)$. Moreover,
$\sigma^{m_0+M} y  = x$, so
\begin{eqnarray}
(L^{m_0+M}g)(x) 
& =      & \sum_{\sigma^{m_0+M}(\xi) = x} e^{f_{m_0+M}(\xi)} g(\xi) \geq e^{f_{m_0+M}(y)} g(y)\\
& \geq & \frac{e^{- (m_0+M) |f|_\infty}}{B_{m_0}}  g(z) . \nonumber
\end{eqnarray}
Keeping $z$ fixed and integrating (3.3) with respect to $x$ gives
$$1 = \int g \, d\nu = \frac{1}{\lambda^{m_0+M}} \int L^{m_0+M} g\; d\nu 
\geq \frac{e^{- (m_0+M) |f|_\infty}}{\lambda^{m_0+M} B_{m_0}}  g(z) .$$
Setting
$K' = B_{m_0} \lambda^{m_0+M} e^{(m_0+M)|f|_\infty} ,$
the above implies $g(z) \leq K'$. This is true for all $z\in \saa$, so 
$|g|_\infty \leq K'$ for all $g \in \Lambda$. Using (3.1), (3.2) and the definition of $B_{m_0}$ we get
$$K' \leq e^{\frac{2\theta}{1-\theta}}(q e^{|f|_\infty})^{m_0+M} e^{(m_0+M)|f|_\infty}
\leq B e^{\frac{\log b}{|\log \theta|} (\log q + 2 |f|_\infty)} < B b^{r_0} = K ,$$
where $K$ is as in (2.3), while $B$ and $r_0$ are defined by (2.4).
(For later convenience we take slightly larger $B$ and $r_0$ than necessary here.)
Thus, 
\be
|g|_\infty \leq K \quad , \quad g \in \Lambda .
\ee

Next, integrating (3.3) with respect to $z$ yields
$$(T^{m_0+M}g)(x) = \frac{1}{\lambda^{m_0+M}} (L^{m_0+M}g)(x) \geq \frac{e^{- (m_0+M) |f|_\infty}}{\lambda^{m_0+M} B_{m_0}}  
= \frac{1}{K'} \geq \frac{1}{K} .$$
Thus, 
\be
\frac{1}{K} \leq \min (T^{m_0+M} g) \quad , \quad g \in \Lambda .
\ee

Let is prove now that $\Lambda$ is an equicontinuous family of functions. Given $\ep > 0$, take
$m \geq m_0$ so that $e^{2\theta^{m-m_0+1}/(1-\theta)} -1 < \ep/K$. Let $x,y\in \saa$ be such that $d_\theta(x,y) \leq \theta^m$. 
Then for any $g \in \Lambda$ we have $g(x) \leq B_m g(y)$, so $g(x) - g(y) \leq (B_m -1) g(y) \leq (B_m-1) K$.
Similarly, $g(y) - g(x) \leq (B_m-1) K$, so 
$$|g(x) - g(y) | \leq (B_m-1) K = \left( e^{2\theta^{m-m_0+1}/(1-\theta)} -1 \right) K < \ep .$$
Hence $\Lambda$ is equicontinuous.

It remains to show that $T(\Lambda) \subset \Lambda$. Let $g \in \Lambda$. Then $Tg \geq 0$
and $\int Tg\, d\nu = 1$. Let $m \geq m_0$ and let $y \in \cc_m[x]$. Given $\xi = (\xi_0, \xi_1, \ldots,) \in \saa$ with
$\sigma \xi = x$, we have $\xi_1 =x_0 = y_0$, $\xi_2 = x_1 = y_1$ , $\ldots, \xi_{m+1} = x_m = y_m$. Set
\be
\eta = \eta(\xi) = (\xi_0, \xi_1, \ldots, \xi_m, \xi_{m+1} = y_m, y_{m+1}, y_{m+2}, \ldots) \in \saa .
\ee
Then $\sigma \eta = y$ and $d(\xi, \eta(\xi)) \leq \theta^{m+1}$, so by (3.2), 
$$|f(\xi) - f(\eta(\xi))| \leq |f|_\theta d_\theta(\xi, \eta(\xi)) \leq b\, \theta^{m+1} \leq \theta^{m-m_0 +1} . $$
This and $g \in \Lambda$ imply $g(\xi) \leq B_{m+1} g(\eta(\xi))$ and
\begin{eqnarray*}
(Tg)(x)
& =    & \frac{1}{\lambda} \sum_{\sigma \xi = x} e^{f(\xi)} g(\xi) \leq 
\frac{1}{\lambda} \sum_{\sigma \xi = x} e^{f(\eta(\xi)) + \theta^{m-m_0 +1} } B_{m+1} g(\eta(\xi))\\
& =    & \frac{e^{\theta^{m-m_0 +1} } B_{m+1} }{\lambda} \sum_{\sigma \eta = y} e^{f(\eta)} g(\eta)
=  e^{\theta^{m-m_0 +1} } e^{2\sum_{j = m-m_0+2}^\infty \theta^j} \;(Tg)(y) \\
& \leq &  B_m (Tg)(y) .
\end{eqnarray*}
Thus, $Tg \in \Lambda$.
\end{proof}

Using the above Lemma and the Schauder-Tychonoff Theorem we derive

\begin{cor} \label{c1}
There exists $h \in \Lambda$ with $Th = h$, i.e. with $Lh = \lambda h$. Moreover we have
$\frac{1}{K} \leq h \leq K$,
where $K$ is given by {\rm (2.3)}.
\end{cor}

The latter follows from (3.4) and (3.5), since $T^{m_0+M} h = h$.

\begin{lem} \label{l2}
{\it There exists a constant $\mu \in (0,1)$ such that for every $g \in \Lambda$ there exists
$\tg \in \Lambda$ with
\be 
T^{m_0+M}g = \mu h + (1-\mu) \tg .
\ee
More precisely we can take}
\be 
\mu = \frac{1-\theta}{4K^2 e^{2\theta/(1-\theta)}} < \frac{1}{4K^2} .
\ee
\end{lem}

\begin{proof}
We use a modification of the proof of Lemma 1.9 in \cite{B}.

Define $\mu$ by (3.8). Given $g \in \Lambda$, set 
$g_1 = T^{m_0+M}g - \mu h$ and $\tg = \frac{g_1}{1- \mu} .$
Then  (3.8) and $T^{m_0+M} g \in \Lambda$ imply
$\mu h \leq \mu K < \frac{1}{K} \leq \min (T^{m_0+M}g) ,$
so $g_1 > 0$. Moreover $\int g_1 \, d\nu = 1-\mu$, so $ \int \tg\, d\nu = 1$.

Next, let $m\geq m_0$ and let $x\in \saa$, $y \in \cc_m[x]$. We will show that $\tg(x) \leq B_m \tg (y)$,
which is equivalent to  $g_1(x) \leq B_m g_1 (y)$, i.e. to
$$(T^{m_0+M}g)(x) - \mu h(x) \leq B_m \left( (T^{m_0+M}g)(y) - \mu h(y)\right) ,$$
that is, to
\be
\mu(B_m h(y) - h(x)) \leq B_m (T^{m_0+M}g)(y) - (T^{m_0+M}g)(x) .
\ee

Given $\xi \in \saa$ with $\sigma \xi = x$ define
$\eta = \eta(\xi)$ by (3.6); then $\sigma\eta = y$ and $\eta \in \cc_{m+1}[\xi]$. For any $G \in \Lambda$,
as in the proof of Lemma \ref{l1}, we have
$$(LG)(x) = \sum_{\sigma\xi = x} e^{f(\xi)} G(\xi) \leq \sum_{\sigma \xi=x} e^{f(\eta) + \theta^{m-m_0+1}} B_{m+1}G(\eta)
\leq e^{\theta^{m-m_0+1}} B_{m+1} (LG)(y) .$$
Using this with $G = T^{m_0+M-1}g = \frac{1}{\lambda^{m_0+M-1}} L^{m_0+M-1}g \in \Lambda$ gives
\be
(T^{m_0+M}g)(x) \leq e^{\theta^{m-m_0+1}} B_{m+1} (T^{m_0+M} g)(y) .
\ee
This and $h(x) \geq \frac{h(y)}{B_m}$ show that to prove (3.9) it is enough to establish
$$\mu \left( B_m - \frac{1}{B_m}\right) h(y) \leq \left( B_m - e^{\theta^{m-m_0+1}} B_{m+1} \right) (T^{m_0+M} g)(y) .$$
Next, the definition of $B_m$, $h(y) \leq K$ and $(T^{m_0+M} g)(y) \geq 1/K$ show that the latter will be true if
we  prove
$$\mu \left( e^{\frac{2\theta^{m-m_0+1}}{1-\theta}} - e^{-\frac{2\theta^{m-m_0+1}}{1-\theta}} \right) 
 \leq \left( e^{2\theta^{m-m_0+1}} B_{m+1}  - e^{\theta^{m-m_0+1}} B_{m+1} \right) \frac{1}{K^2} ,$$
 which is equivalent to
 \be
 \mu \left( e^{\frac{2\theta^{m-m_0+1}}{1-\theta}} - e^{-\frac{2\theta^{m-m_0+1}}{1-\theta}} \right) 
 \leq e^{\theta^{m-m_0+1} + \frac{2\theta^{m-m_0+2}}{1-\theta} } \cdot \frac{e^{\theta^{m-m_0+1}}-1}{K^2} .
 \ee
For the left-hand-side of (3.11) there exists some $z$ with $|z| < 2\theta^{m-m_0+1}/(1-\theta) $ such that 
$$ \mu \left( e^{\frac{2\theta^{m-m_0+1}}{1-\theta}} - e^{-\frac{2\theta^{m-m_0+1}}{1-\theta}} \right) 
= \mu e^z \frac{4 e^{\theta^{m-m_0+1}}}{1-\theta} \leq \mu \frac{4 e^{\frac{2\theta}{1-\theta}}}{1-\theta} \theta^{m-m_0+1} .$$
For the right-hand-side of (3.11) we have
$$ e^{\theta^{m-m_0+1} + \frac{2\theta^{m-m_0+2}}{1-\theta} } \cdot \frac{e^{\theta^{m-m_0+1}}-1}{K^2} >
\frac{e^{\theta^{m-m_0+1}}-1}{K^2} \geq \frac{\theta^{m-m_0+1}}{K^2} .$$
Thus, (3.11) would follow from
$\mu \frac{4 e^{\frac{2\theta }{1-\theta}}}{1-\theta} \theta^{m-m_0+1} \leq \frac{\theta^{m-m_0+1}}{K^2} .$
The latter is clearly true by (3.8). This proves (3.11) which, as we observed, implies (3.9). Hence
$\tg(x) \leq B_m \tg (y)$ which shows that $\tg \in \Lambda$.
\end{proof}

\begin{lem} \label{l3}
{\it There exist constants $A > 0$ and  $\beta \in (0,1)$ such that 
\be
|T^ng - h|_\infty \leq A \beta^n 
\ee
for every $g \in \Lambda$ and every integer $n \geq 0$.  More precisely we can take}
\be
A =  4K^2 \quad , \quad  \beta =  1 - \frac{1-\theta}{4K^3}  \in  (\theta ,1) .
\ee
\end{lem}

\begin{proof}
We use a modification of the proof of Lemma 1.10 in \cite{B}.

Let $g \in \Lambda$. Given an integer $n \geq 0$ write $n = p(m_0+M) + r$ 
for some integers $ p \geq 0$ and $r = 0,1, \ldots, m_0+M-1$. By Lemma \ref{l2},
$T^{m_0+M}g = \mu h + (1-\mu) g_1$ for some $g_1 \in \Lambda$. Similarly,
$T^{m_0+M}g_1 = \mu h + (1-\mu) g_2$ for some $g_2 \in \Lambda$, so
$$T^{2(m_0+M)}g = \mu h + (1-\mu) (\mu h + (1-\mu) g_2)
= \mu h (1 + (1-\mu)) + (1-\mu)^2 g_2 .$$
Continuing in this way we prove by induction
$$T^{p(m_0+M)}g  = \mu h (1 + (1-\mu) + \ldots + (1-\mu)^{p-1}) + (1-\mu)^p g_p$$
for some $g_p \in \Lambda$. Thus,
$$T^{p(m_0+M)}g  = \mu h \frac{1 - (1-\mu)^p}{1 - (1-\mu)} +  (1-\mu)^p g_p
=  h (1 - (1-\mu)^p) +  (1-\mu)^p g_p , $$
and therefore, using (3.4),
\be
|T^{p(m_0+M)}g - h|_\infty \leq (1-\mu)^p |h - g_p|_\infty \leq 2K (1-\mu)^p .
\ee

Next, notice that by (3.1) for every bounded function $G$ on $\saa$ we have
$$|(TG)(x)| = \frac{1}{\lambda} \left| \sum_{\sigma \xi = x} e^{f(\xi)} G(\xi) \right|
\leq  \frac{q e^{|f|_\infty}}{\lambda} |G|_\infty \leq  q e^{2|f|_\infty} |G|_\infty ,$$
so $|TG|_\infty \leq q e^{2|f|_\infty} |G|_\infty$. Using this $r$ times 
and setting
$\beta' = (1-\mu)^{\frac{1}{m_0+M}} ,$
 yields
\begin{eqnarray*}
|T^ng - h|_\infty
& =     & |T^r (T^{p(m_0+M)}g-h)|_\infty \leq (q e^{2|f|_\infty})^r |T^{p(m_0+M)}g - h|_\infty\\
& \leq & 2K (q e^{2|f|_\infty})^{m_0+M} \frac{(\beta')^n}{(\beta')^r}
\leq \frac{ 2K}{1-\mu} (q e^{2|f|_\infty})^{m_0+M} (\beta')^n .
\end{eqnarray*}

As in previous estimates, using (3.2) and (3.4) we get
\begin{eqnarray*}
(q e^{2|f|_\infty})^{m_0+M} 
& =    & q^M e^{2M|f|_\infty} e^{m_0 (\log q + 2|f|_\infty)}
\leq q^M e^{2M|f|_\infty} e^{(\frac{\log b}{|\log \theta|}+1) (\log q + 2|f|_\infty)}\\
& \leq & q^{M+1} e^{(2M+1)|f|_\infty} b^{r_0} < K .
\end{eqnarray*}
We have $1-\mu \geq 1/2$ by (3.8), so the above  and (3.13)  imply $|T^ng - h|_\infty \leq A (\beta')^n$. 

It remains to show that $\beta' \leq \beta$. We will use the elementary inequality $(1-x)^a \leq 1 - ax$
for $0 \leq x < 1$ and $0 < a < 1$. It implies
$\beta' = (1-\mu)^{\frac{1}{m_0+M}}  \leq 1- \frac{\mu}{m_0+M}  < 1- \frac{\mu}{e^{m_0+M}} 
= 1 - \frac{1-\theta}{4K^2 e^{m_0+M} e^{2\theta/(1-\theta)}} <  1 - \frac{1-\theta}{4K^3} = \beta .$
This proves the lemma.
\end{proof}

\begin{lem} \label{l4}
{\it For every $g \in \Lambda$ we have $|g|_\theta < B b K$, and so $\|g\|_\theta < 2 B b K$.}
\end{lem}

\begin{proof}
Let $g \in \Lambda$ and let $x,y\in \saa$ be such that $d_\theta(x,y) = \theta^m$.
If $m \leq m_0-1$, then by (3.4),
$$|g(x)-g(y)| \leq 2K = 2K \frac{d_\theta(x,y)}{\theta^m} \leq \frac{2K}{\theta^{m_0-1}} d_\theta (x,y)
 \leq 2 b K d_\theta (x,y)  \leq B b K d_\theta (x,y) .$$
Next, assume that $m \geq m_0$. Then using again (3.2) and (3.4) we get
$$B_m - 1 = e^{\frac{2\theta^{m-m_0+1}}{1-\theta}} -1 \leq e^{2\theta/(1-\theta)} \frac{2\theta^{m-m_0+1}}{1-\theta}
= \frac{2e^{2\theta/(1-\theta)}}{(1-\theta)\theta^{m_0-1}} \theta^{m}
 < B\, b \, \theta^m  .$$
Since, $g(x) \leq B_m g(y)$ we have $g(x) - g(y) \leq (B_m -1) g(y) \leq (B b \theta^m) \, K = B b K d_\theta(x,y)$.
Similarly, $g(y) - g(x) < B b K d_\theta(x,y)$, so $|g(x) - g(y)| < B b K  d_\theta(x,y)$.
\end{proof}

In particular, $\Lambda \subset \ff_\theta(\saa)$, so $\lambda$ is an eigenvalue of the transfer operator 
$L_f : \ff_\theta(\saa) \longrightarrow \ff_\theta(\saa)$ and $h > 0$ is a corresponding
eigenfunction. Moreover, following arguments from the proof of Theorem 2.2 in \cite{PP}, one proves
that $\lambda$ is a simple eigenvalue and $\spec_\theta (L_f) \subset \{ z : |z| \leq \lambda\}$.
Also, following the argument from the proof of Theorem 1.5 in \cite{Ba}, one shows that
the essential spectral radius of $L_f$ as an operator on $\ff_\theta(\saa)$ is $\theta\, \lambda$.

\begin{lem} \label{l5}
{\it For every $g \in \ff_\theta(\saa)$ we have 
\be
\left| \frac{1}{\lambda^n}\, L^n g - h\, \int g\, d\nu\right|_\infty \leq A_1\, \beta^n \, 
\|g\|_\theta \quad , \quad  n \geq 0 ,
\ee
where} $A_1 = 2A b  = 8K^2 b.$
\end{lem}
\begin{proof}
We will proceed as in \cite{St1} with some modifications.

Let $g\in \ff_\theta(\saa)$. First, assume that $g \geq 0$. The case  $|g|_\theta = 0$ follows trivially from Lemma \ref{l3}, 
so assume $|g|_\theta > 0$ and set $\tg = C\, g + 1$, where $C = \frac{2}{(1-\theta) b |g|_\theta}$.
Then $\omega = \int \tg\, d\nu \geq 1$.

We will check that $\tg/\omega \in \Lambda$. Let $m \geq m_0$,
and let $x,y\in \saa$ be such that $y \in \cc_m[x]$. Assume e.g. $\tg(x) \geq \tg(y)$. We have
$$\tg(x) - \tg(y) = C(g(x) - g(y)) \leq C |g|_\theta d_\theta(x,y) = C|g|_\theta \theta^{m} .$$
Hence, using $\tg (y) \geq 1$ and (3.2) it follows that
\begin{eqnarray*}
\tg(x)
& \leq & \tg(y) + \frac{C|g|_\theta \theta^{m-m_0+1}}{\theta^{m_0-1}} \leq
\tg(y)\left( 1  + \frac{2\theta^{m-m_0+1}}{(1-\theta) b\, \theta^{m_0-1}}\right)\\
& \leq & \tg(y) \left( 1  + \frac{2\theta^{m-m_0+1}}{(1-\theta) }\right) \leq  \tg(y) e^{\frac{2\theta^{m-m_0+1}}{(1-\theta) }}
= \tg(y) \, B_m .
\end{eqnarray*}
This shows that $\tg/\omega \in \Lambda$, and by (3.12), $|T^n\tg - \omega h|_\infty \leq A \omega \beta^n$.
Thus,
$$\left|T^n(Cg + 1) - h \left(C \int g\, d\nu + 1\right) \right|_\infty \leq A \omega \beta^n .$$
Using this and (3.12) with $g = 1$ yields 
$$C \left|T^ng - h \int g\, d\nu \right|_\infty \leq \left|T^n 1 - h \int 1\, d\nu \right|_\infty 
+ A \omega \beta^n \leq A (\omega + 1) \beta^n ,$$
so
$\left|T^ng - h \int g\, d\nu \right|_\infty \leq  A \frac{\omega + 1}{C} \beta^n .$ Finally,
$$\frac{1}{C} (\omega+1) = \int g\, d\nu + \frac{2}{C} \leq |g|_\infty + (1-\theta) b |g|_\theta
\leq  b \|g\|_\theta .$$
Hence 
\be
\left|T^ng - h \int g\, d\nu \right|_\infty \leq  A b \|g\|_\theta \beta^n .
\ee

For general $g\in \ff_\theta(\saa)$, write $g = g_+ - g_-$, where  $g_+ = \max\{ g, 0\} \geq 0$ and
$g_- = g_+ - g \geq 0$. Then  $g_+, g_- \in \ff_\theta(\saa)$,  $\|g_+\|_\theta \leq \|g\|_\theta$
$\|g_-\|_\theta \leq \| g \|_\theta$, and $g_+\leq |g|_\infty$, $g_-\leq |g|_\infty$, so 
$\|g_+\|_\theta \leq \|g\|_\theta$ and $\|g_-\|_\theta \leq \|g\|_\theta$. Using (3.16) for $g_+$
and $g_-$ implies $\left|T^ng - h \int g\, d\nu \right|_\infty \leq  2A b \|g\|_\theta \beta^n$.
\end{proof}

We will now sketch the proofs of the Basic Inequalities (see Proposition 2.1 in \cite{PP} or Lemma 1.2 in \cite{B})
keeping track on the constants involved. We continue to use the notation from Sect. 2 and also the one
introduced above for the function $f$ and the operator $L = L_f$.

\begin{lem} \label{l6}
(Basic Inequalities) {\it We have
\be
|L^ng|_\infty \leq K^2 \lambda^n |g|_\infty \quad, \quad g \in C(\saa)\:,\: n \geq 0 ,
\ee
and
\be
|L^ng|_\theta \leq K^2 \lambda^n \left[\frac{2|f|_\theta}{1-\theta}  |g|_\infty  + \theta^n |g|_\theta \right]
\quad, \quad g \in \ff_\theta(\saa) \:,\: n \geq 0 .
\ee
Consequently,}
\be
\|L^ng\|_\theta \leq \frac{4 b K^2}{1-\theta}  \lambda^n  \|g\|_\theta 
\quad, \quad g \in \ff_\theta(\saa) \:,\: n \geq 0 .
\ee
\end{lem}

\begin{proof}
We will just follow the standard arguments to derive the above estimates.

It follows from Corollary \ref{c1}  that
$$L^n 1 = K\, (L^n 1/K) \leq K\, L^n h = K \lambda^n h\leq K^2 \lambda^n ,$$
so $L^n 1 \leq K^2 \lambda^n$ for all $n \geq 0$.

Given $g\in C(\saa)$, for any $x\in \saa$ and any $n \geq 1$ we have
$$|(L^ng)(x)| \leq \sum_{\sigma^n\xi = x} e^{f_n(\xi)}  |g(\xi)| \leq |g|_\infty \, (L^n1)(x)
\leq K^2 \lambda^n |g|_\infty .$$
This proves (3.17).

Next, let $g \in \ff_\theta(\saa)$, and let $ n \geq 1$. Given $x\in \saa$ and $y \in \cc_n[x]$,
for any $\xi \in \saa$ with $\sigma^n\xi = x$, denote by $\eta = \eta(\xi)$ the unique element
of $\saa$ such that $\sigma^n \eta = y$ and $d_\theta(\xi, \eta) = \theta^{n} \, d_\theta(x,y)$. 
Then
$$|f_n(\xi) - f_n(\eta(\xi))| \leq \sum_{j=0}^{n-1} |f(\sigma^j \xi) - f(\sigma^j \eta)|
\leq \sum_{j=0}^{n-1} |f|_\theta \theta^{n-j} d_\theta(x,y) \leq \frac{|f|_\theta}{1-\theta} \, d_\theta (x,y) ,$$
and therefore
$$\left| e^{f_n(\xi)} - e^{f_n(\eta)}\right| \leq |f_n(\xi) - f_n(\eta)| \, e^{\max\{ f_n(\xi), f_n(\eta)\}}
\leq \frac{|f|_\theta}{1-\theta} \, d_\theta (x,y) \left[e^{f_n(\xi)} + e^{f_n(\eta)}\right] .$$
The above yields
\begin{eqnarray*}
&        & |(L^n g)(x) - (L^n g)(y)|
 \leq  \sum_{\sigma^n \xi = x} \left| e^{f_n(\xi)} g(\xi) - e^{f_n(\eta(\xi))} g(\eta(\xi)) \right| \\
& \leq & \sum_{\sigma^n \xi = x} \left[\left| e^{f_n(\xi)} - e^{f_n(\eta)} \right|\, |g(\xi)| + e^{f_n(\eta)} |g(\xi) -  g(\eta)| \right] \\
& \leq & \frac{|g|_\infty |f|_\theta}{1-\theta} \, d_\theta (x,y) \sum_{\sigma^n \xi = x} \left[e^{f_n(\xi)} + e^{f_n(\eta)}\right] + 
|g|_\theta \theta^n d_\theta (x,y) \sum_{\sigma^n \xi = x} e^{f_n(\eta)}\\
& \leq & \frac{|g|_\infty |f|_\theta}{1-\theta} \, d_\theta (x,y) [ (L^n 1)(x) + (L^n 1)(y)] + |g|_\theta \theta^n d_\theta (x,y)\, (L^n 1)(y)\\
& \leq & K^2 \lambda^n \left[  \frac{2  |f|_\theta}{1-\theta} |g|_\infty + \theta^n |g|_\theta \right] \, d_\theta (x,y) ,
\end{eqnarray*}
which proves (3.18). The latter obviously implies (3.19).
\end{proof}

To derive part (c) in   Theorem \ref{t1}, just notice that (2.5) implies $\rho > \beta$, where $\beta$ is given by (3.13). If
 $z\in \spec_\theta (L_f)$ with $\rho\, \lambda < |z|$ and $z \neq \lambda$, then
$z$ is an eigenvalue of $L$. If $g$ is a corresponding eigenfunction, then 
$\di \int g\, d\nu = 0$ by (3.15), and using (3.15) again gives $|z| \leq \beta\, \lambda < \rho\, \lambda$, a contradiction.
This shows that $\spec_\theta (L_f)\cap \{ z : |z| > \rho\, \lambda\} = \{ \lambda\}$.

We will now use (3.15) to prove

\begin{lem} \label{l7}
{\it For every $g \in \ff_\theta(\saa)$ we have 
\be
\left\| \frac{1}{\lambda^n}\, L_f^n g - h\, \int g\, d\nu\right \|_\theta \leq A_2\, \rho^n \, 
\|g\|_\theta \quad , \quad  n \geq 0 ,
\ee
where $\rho$ is given by {\rm (2.5)} and } $A_2 = \frac{100K^5 b^3}{(1-\theta)}  .$
\end{lem}

\begin{proof}
We will again use a corresponding argument in \cite{St1} with some modifications.
Let $g \in \ff_\theta(\saa)$ and let $n \geq 1$.

\ms

\noindent
{\bf Case 1.}  $\di \int g\, d\nu = 0$. Set $C = \frac{2|f|_\theta}{1-\theta}$, $\ell = [n/2]$ and $k = n-\ell$.
First notice that in the present case (3.15) gives
$|L^\ell g|_\infty \leq A_1 \lambda^\ell \beta^\ell \|g\|_\theta .$
Using this, (3.18), (3.15) and $\theta \leq \beta$ yields
\begin{eqnarray*}
|L^n g|_\theta
&  =   & |L^k(L^\ell g)|_\theta \leq K^2 \lambda^k(C  |L^\ell g|_\infty +  \theta^k\, |L^\ell g|_\theta)\\
& \leq & K^2 \lambda^k \, \left[ C A_1\, \lambda^\ell\, \beta^\ell \|g\|_\theta + \theta^k K^2 \lambda^\ell(C |g|_\infty +  \theta^\ell\,|g|_\theta) \right]
 \leq  A'\,\lambda^n  \beta^{n/2}\, \|g\|_\theta ,
\end{eqnarray*}
where $A' = \frac{40 K^4 b^2}{1-\theta}$. This proves (3.20) in the case considered. 

\def\tD{\tilde{D}}

\ms

\noindent
{\bf Case 2.} General case. Let $g \in  \ff_\theta(\saa)$ and let $n \geq 1$.
Set $\tg = g - \alpha \, h$, where $\alpha = \int g\, d\nu$. Then $\int \tg \, d\nu = 0$, so by Case 1,
we have
$|L^n \tg|_\theta \leq A'\,\lambda^n  \beta^{n/2}\, \|\tg\|_\theta .$
By Corollary \ref{c1} we have $|\tg|_\infty \leq |g|_\infty + K|g|_\infty \leq (1 + K) \|g\|_\theta$, while Lemma \ref{l4} implies
$|\tg|_\theta \leq |g|_\theta + |g|_\infty |h|_\theta \leq B b K \|g\|_\theta$. Thus, $\|\tg\|_\theta \leq 2B b K \|g\|_\theta$.
This and the above estimate imply
\begin{eqnarray*}
\left| \frac{1}{\lambda^n}\, L^n g - h\, \int g\, d\nu\right |_\theta
& =     & \frac{1}{\lambda^n} |L^n (g- \alpha h)|_\theta = \frac{1}{\lambda^n} |L^n\tg|_\theta\\
& \leq & A'\,\beta^{n/2}\, \|\tg\|_\theta \leq A'\,2 B b K  \beta^{n/2}\, \|g\|_\theta .
\end{eqnarray*}
Combining with (3.15) gives
$$\left\| \frac{1}{\lambda^n}\, L^n g - h\, \int g\, d\nu\right\|_\theta \leq \frac{100 B K^5 b^3}{1-\theta} \beta^{n/2}\, \|g\|_\theta . $$
Finally it follows from $\sqrt{1-x} < 1 - x/2$ for $0 < x < 1$ and (2.5) that
$\sqrt{\beta} = \sqrt{1- \frac{1-\theta}{4K^3}} \leq 1- \frac{1-\theta}{8K^3} = \rho .$
This proves (2.6).
\end{proof}


\footnotesize





\begin{thebibliography}{10}


\bibitem{Ba} V. Baladi, \emph{ Positive transfer operators and decay of correlations}. World Scientific,
Singapore, 2000. 

\bibitem{B} R. Bowen, \emph{ Equilibrium states and the ergodic theory of Anosov diffeomorphisms},
Lecture Notes in Mathematics {\bf 470}, Springer-Verlag, Berlin, 1975.

\bibitem{D} D. Dolgopyat,  {On decay of correlations in Anosov flows}.
{\em Ann. of Math.} {\bf 147} (1998), 357-390.

\bibitem{N} F. Naud, \emph{ Dynamics on Cantor sets and analytic properties of zeta functions}.  PhD Thesis,
University of Bordeaux I, 2003.

\bibitem{PP} W. Parry and M. Pollicott, { Zeta functions and  the periodic orbit
structure of hyperbolic dynamics}.  Ast\'erisque {\bf 187-188}, 1990.

\bibitem{Po} M. Pollicott,  {A complex Ruelle operator theorem
and two counterexamples}. Ergod. Th. \& Dynam. Sys. {\bf 4} (1984), 135-146.


\bibitem{R1} D. Ruelle, { Statistical mechanics of an one-dimensional lattice gas},
Commun. Math. Phys.  {\bf 9} (1968), 267-278. 

\bibitem{R2} D. Ruelle,  { A measure associated with Axiom A attractors}.
Amer. J. Math. {\bf 98} (1976), 619-654. 

\bibitem{St1} L. Stoyanov, { On the Ruelle-Perron-Frobenius Theorem}. Asympt. Analysis {\bf 43} (2005), 131-150.

\bibitem{St2} L. Stoyanov, \emph{ Scattering resonances for several small convex bodies and the Lax-Phillips
conjecture}. Memoirs Amer. Math. Soc. Vol. {\bf 199} (2009).

\bibitem{St3} L. Stoyanov, { Spectra of Ruelle transfer operators for Axiom A flows}.
{\em Nonlinearity}, {\bf 24} (2011), 1089-1120.



\end{thebibliography}
\end{document}